\documentclass[12pt,reqno]{amsart}

\usepackage[all]{xy}
\usepackage{hyperref}

\newtheorem{theorem}{Theorem}[section]
\newtheorem{proposition}[theorem]{Proposition}
\newtheorem{lemma}[theorem]{Lemma}

\theoremstyle{definition}
\newtheorem{definition}[theorem]{Definition}
\newtheorem{remark}[theorem]{Remark}

\numberwithin{equation}{section}

\newcommand{\C}{\mathbb{C}}

\begin{document}

\title[Uniformization of branched surfaces and Higgs bundles]{Uniformization
of branched surfaces and Higgs bundles}

\author[I. Biswas]{Indranil Biswas}

\address{School of Mathematics, Tata Institute of Fundamental Research,
Homi Bhabha Road, Mumbai 400005, India}

\email{indranil@math.tifr.res.in}

\author[S. Bradlow]{Steven Bradlow}

\address{Department of Mathematics, University of Illinois at Urbana-Champaign,
Urbana, IL 61801, USA}

\email{bradlow@math.uiuc.edu}

\author[S. Dumitrescu]{Sorin Dumitrescu}

\address{Universit\'e C\^ote d'Azur, CNRS, LJAD, France}

\email{dumitres@unice.fr}

\author[S. Heller]{Sebastian Heller}

\address{Institute of Differential Geometry,
Leibniz Universit\"at Hannover,
Welfengarten 1, 30167 Hannover, Germany}

\email{seb.heller@gmail.com}



\date{}

\begin{abstract}
Given a compact connected Riemann surface $\Sigma$ of genus $g_\Sigma\, \geq\, 2$, and an effective divisor
$D\, =\, \sum_i n_i x_i$ on $\Sigma$ with $\text{degree}(D)\, <\, 2(g_\Sigma -1)$, there is a
unique cone metric on $\Sigma$ of constant negative curvature $-4$ such that the cone angle at
each point $x_i$ is $2\pi n_i$ \cite{McO,Tr}. We describe the Higgs bundle on $\Sigma$ corresponding to the
uniformization associated to this conical metric. We also give a family of Higgs bundles on $\Sigma$
parametrized by a nonempty open subset of
$H^0(\Sigma,\,K_\Sigma^{\otimes 2}\otimes{\mathcal O}_\Sigma(-2D))$ that correspond to conical
metrics of the above type on moving Riemann surfaces. These are inspired by Hitchin's results
in \cite{Hi} for the case $D\,=\, 0$.
\end{abstract}

\maketitle

\section{Introduction}

In one of the earliest applications of Higgs bundles, Hitchin exploited the correspondence between their stability 
(in the sense of Geometric Invariant Theory) and the existence of solutions to natural gauge-theoretic equations to 
construct smooth hyperbolic structures, i.e., Riemannian metrics with constant negative curvature, on a closed 
oriented surface (see \cite{Hi}), and thereby to parameterize Teichm\"uller space by the holomorphic quadratic 
differentials on a fixed Riemann surface. This approach has since been adapted by others (see below) 
to produce hyperbolic structures with conical singularities, thereby yielding a Higgs bundle proof of 
the uniformization theorem with prescribed conical singularities by McOwen and Troyanov \cite{McO,Tr}. These 
constructions of singular metrics are based on Higgs bundles with parabolic structures or, equivalently, filtered 
Higgs bundles in the sense of Simpson \cite{Si}. In this article we describe how, under suitable conditions, 
singular hyperbolic metrics can be obtained from ordinary Higgs bundles.

The Higgs bundles in Hitchin's parameterization of Teichm\"uller space, and also in the above constructions of 
singular hyperbolic structures, are pairs $(E,\,\Phi)$, where $E$ is a rank two holomorphic bundle on a Riemann 
surface $\Sigma$ of genus $g_\Sigma$, and the Higgs field $\Phi$ is a twisted endomorphism of it (see Section 
\ref{sec:Higgs}). Such pairs define ${\rm SL}(2,\C)$-Higgs bundles. In fact they are additionally constrained so 
that the structure group of the bundles reduce to ${\rm SO}(2,\C)$ and the Higgs fields are symmetric with respect 
to the orthogonal structure. This means that they are determined by data sets of the form $(L,\,\alpha,\,\delta)$, 
where $L$ is a holomorphic line bundle on $\Sigma$ and $\alpha,\, \delta$ (the components of the Higgs field) are 
respectively holomorphic sections of $L^2\otimes K_{\Sigma}$ and $L^{-2}\otimes K_{\Sigma}$ (see Section 
\ref{sec:Higgs}). As such they are properly viewed as ${\rm SL}(2,\mathbb R)$-Higgs bundles or in other words, under 
the non-abelian Hodge correspondence, the Higgs bundles correspond to surface group representations in the split 
real form ${\rm SL}(2,\mathbb R)\,\subset\, {\rm SL}(2,\mathbb C)$.

From the perspective of the Higgs bundles, the results in this paper address the following question about 
$\mathrm{SL}(2,\mathbb R)$-Higgs bundles on a closed Riemann surface $\Sigma$. It is well known that such Higgs 
bundles admit a moduli space which has multiple connected components. Among these are $2^{2g}$ connected components 
called Hitchin components, where the underlying holomorphic
line bundle $L$ in the triple $(L,\,\alpha,\,\delta)$ is a square root 
of the canonical bundle $K_{\Sigma}$. The Higgs bundles in these components are precisely those which appear in 
Hitchin's parameterization of hyperbolic structures on $\Sigma$. The other components are labelled by an integer 
$\tau$, equal to degree of $L$, which satisfies the so-called Milnor-Wood bound $|\tau|<g-1$ \cite{Go,Hi}. The 
question we address is whether these other components parameterize geometric structures analogous to the hyperbolic 
structures parameterized by the Hitchin components. The answer we give here applies to a certain subset in each 
component, for which we show that it parameterizes hyperbolic structures with conic singularities; see Theorem 
\ref{thm1}.

The starting point for our results is the fact that if $(L,\,\delta,\,\alpha)$ is a stable ${\rm SL}(2,\mathbb R)$-Higgs 
bundle or, equivalently, if
$$
\left(L\oplus L^{-1},\,\begin{pmatrix}0& \alpha\\ \delta&0\end{pmatrix}\right)
$$
is a stable ${\rm SL}(2,\mathbb R)$-Higgs bundle, then 
$L$ admits a hermitian metric which satisfies the Higgs bundle equations. The metric may be described as an
equivariant harmonic map from $\widetilde{\Sigma}\,=\,\mathbb D$, the universal cover of $\Sigma$, to the hyperbolic plane
$\mathbb H\, :=\,{\rm SL}(2,\mathbb R)/{\rm SO}(2)$.

If ${\rm degree}(L)\,=\,g-1$, then $L^{\otimes 2}\,=\,K_{\Sigma}$ and Hitchin showed in \cite{Hi} how the Higgs 
bundle metric combines with the Higgs field to define a hyperbolic structure, i.e., a constant negative curvature 
metric on $K^{-1}_{\Sigma}$. Our key observation, based on considerations of the equivariant harmonic maps, is that 
if $L^{\otimes 2}\,=\, K_\Sigma\otimes {\mathcal O}_\Sigma(-D)$ , where $D$ is an effective divisor on $\Sigma$, 
then the analogous construction yields a symmetric 2-tensor on the complement $\Sigma\setminus D$ which defines a 
hyperbolic structure wherever
the 2-tensor is non-degenerate (see Proposition \ref{h2subseth3}). Moreover at the points in $D$ 
this metric has conic singularities with cone angles given by the multiplicities of the
points in $D$. The main part of our proof thus 
involves showing the non-degeneracy of the metric away from $D$ under natural conditions.

The condition $L^{\otimes 2}\,=\, K_\Sigma\otimes {\mathcal O}_\Sigma(-D)$ requires that ${\rm degree}(D)$ is even.
We also address the case where ${\rm degree}(D)$ is odd (see Section \ref{sec: odd degree}) --- in which case it is not possible to satisfy $L^{\otimes 2}\,=\, K_\Sigma\otimes {\mathcal O}_\Sigma(-D)$ --- by introducing an unramified double cover of $\Sigma$.

Our methods are closely related to other gauge-theoretic approaches to singular hyperbolic structures. Some (see 
\cite{KW, M, NS,BGG}) use rank two parabolic Higgs bundles or, in the case of \cite{NS}, equivalent orbifold Higgs 
bundles on orbifold quotients of a smooth surface, while \cite{Ba} is based on the abelian vortex equations.

The Higgs bundles are parabolic (or orbifold) versions of the ${\rm SL}(2,\mathbb{R})$-Higgs bundles in the Hitchin 
component, while our results extend to Higgs bundles in components with non-maximal $\tau$. All of these 
gauge-theoretic constructions start with a fixed smooth hyperbolic structure on the underlying real surface of 
$\Sigma$. In \cite{KW, M} and \cite{Ba} the new singular hyperbolic structures are in the same conformal class as 
the background metric, while our construction covers a subset in Teichm\"uller space containing 
$H^0(\Sigma,\,K_\Sigma^{\otimes 2}\otimes{\mathcal O}_\Sigma(-2D-D_{\rm red}))$ (see Theorem \ref{thm1}). On the 
other hand, the cone angles at the singularities in our metrics are all integer multiples of $2\pi$, while the 
hyperbolic structures coming from the parabolic Higgs bundles in \cite{NS} are rational multiples of $2\pi$ and in 
\cite{KW} they can have arbitrary small cone angles. The hyperbolic structures constructed in \cite{BGG} are on a 
punctured Riemann surface. The metrics are not necessarily in the conformal class of the uniformizing metric on the 
Riemann surface (in other words, in the Higgs field $\alpha\,\neq\, 0$ is allowed) but the asymptotics at
the punctures are not considered there.

Baptista (in \cite{Ba}) uses solutions of the vortex equations on line bundles to construct 
singular metrics on the Riemann surface, with singularities at the zeros of the vortex section. 
The new metrics are in the same conformal class as the metric on the Riemann surface. Starting 
with a hyperbolic structure on the Riemann surface, then new singular metrics define hyperbolic 
structures away from their singularities and have conical singularities with integral cone 
angles (determined by the order of vanishing of the section). Though no details are given, 
Baptista asserts that non-integral cone angles can be achieved by using vortices with parabolic 
singularities or by starting with singular hyperbolic structures, as in \cite{BB13}. In fact, 
the abelian vortices used by Baptista can be viewed as special cases of the ${\rm SL}(2,\mathbb 
R)$-Higgs bundles, namely those which are fixed by a well-known $\C^*$-action on the Higgs 
bundle moduli space. From this perspective, we recover the Baptista hyperbolic structures as 
special cases of our results.

We note finally that Mandelbaum introduced the notion of a branched complex projective structure on a Riemann 
surface with branching over an effective divisor \cite{Ma1} and \cite{Ma2}. The cone metrics considered here define 
branched complex projective structures in the sense of Mandelbaum. More precisely, the cone metric on $\Sigma$ of 
constant negative curvature $-4$ such that the cone angle at each $x_i$ is $2\pi n_i$ (where $D=\Sigma_i n_ix_i$ is 
the divisor mentioned above) produces a branched complex projective structure on $\Sigma$ with branching divisor 
$\sum_i n_i x_i$.

\section{The Higgs bundles} \label{sec:Higgs}

\subsection{${\rm SL}(2,\mathbb C)$ and ${\rm SL}(2,\mathbb R)$ Higgs bundles}

\begin{definition}\label{def1}
Let $\Sigma$ be a compact connected Riemann surface of genus $g_\Sigma$, 
with $g_\Sigma\, \geq\, 2$. An ${\rm SL}(2,\mathbb C)$-Higgs bundle on $\Sigma$ is a pair 
$(E,\,\Phi)$, where $E$ is a rank 2 holomorphic bundle with trivial determinant on $\Sigma$
(meaning $\bigwedge^2 E\,=\, {\mathcal O}_\Sigma$) and 
$\Phi$ is holomorphic section of the vector bundle $End(E)\otimes K_{\Sigma}$, in
other words, $\Phi$ is a ${\mathcal O}_\Sigma$--linear map
$$
\Phi\,:\,E\,\longrightarrow \,E\otimes K_{\Sigma}\, .
$$
Here $K_{\Sigma}$ denotes the canonical bundle on $\Sigma$. The Higgs bundle $(E,\,\Phi)$ is 
called {\it semistable} (respectively, {\it stable}) if
$$
{\rm degree}(E') \,\leq\, 0 \ \ \text{(respectively,\ ${\rm degree}(E') \,<\, 0$)}
$$
for all invariant holomorphic line subbundles $E'\, \subset\, E$ that satisfy the
condition $\Phi(E')\, \subset\, E'\otimes K_{\Sigma}$.
The space of all $S$-equivalence classes of semistable Higgs bundles form a moduli space which
we denote by $\mathcal{M}({\rm SL}(2,\C))$.
\end{definition}

Here we are interested in the ${\rm SL}(2,\mathbb C)$-Higgs bundles of the form
\begin{equation}\label{SL2CHigss}
\left(E\,=\,L\oplus L^{-1},\;
\Phi\,=\, \begin{pmatrix}0& \alpha\\ \delta&0\end{pmatrix}\right)\, ,
\end{equation}
where $L$ is a holomorphic line bundle with ${\rm degree}(L)\,\ge \,0$, $\alpha\,\in\, H^0(\Sigma,\, L^2
\otimes K_{\Sigma})$ and $\delta\,\in \,H^0(\Sigma,\, L^{-2}\otimes K_{\Sigma})$.

Notice that the vector bundle $E$ in \eqref{SL2CHigss} admits an orthogonal structure defined by the non-degenerate form 
$Q\,=\,\begin{bmatrix}0&1\\1&0\end{bmatrix}$. The structure group of $E$ thus reduces from ${\rm SL}(2,\mathbb C)$ to 
${\rm SO}(2,\mathbb C)\,\simeq\,\mathbb C^*$. Equivalently we may say that $E$ is a rank two bundle associated to the
principal $\mathbb C^*$--bundle $L\setminus 0_\Sigma$ for the representation
\begin{equation}\label{el1}
\mathbb C^*\,\hookrightarrow \,{\rm SL}(2,\mathbb C)
\end{equation}
given by $z\,\longmapsto\, 
\begin{bmatrix}z&0\\0&z^{-1}\end{bmatrix}$. Furthermore, the Higgs field is symmetric with respect to
the form $Q$, or in other words,
$$Q(\Phi(v),\, w)\, =\, Q (v,\, \Phi(w))\, \in\, (K_\Sigma)_x$$ for all $v,\, w\, \in\, E_x$ and
all $x\, \in\, \Sigma$. As such, the defining data set for $(E,\,\Phi)$, i.e., $(L,\, (\alpha,\,
\delta))$, produces a Higgs bundle for the real form ${\rm SL}(2,\mathbb R)\,\subset\, {\rm SL}(2,\C)$.

The general definition of a principal $G$-Higgs bundle, where $G\,\subset\, G_{\mathbb C}$ is a real form of a reductive 
complex Lie group $G_{\mathbb C}$, can be found in \cite{BrGG}. The definition requires the following structures:
\begin{enumerate} 
\item A maximal compact subgroup $H\,\subset\, G$ and its complexification $H_{\mathbb C}$, and

\item the Cartan 
decomposition of the complex Lie algebra ${\rm Lie}(G_{\mathbb C})\,=\,{\rm Lie}(H_{\mathbb C})
\oplus \mathfrak{m}_{\mathbb C}$, which is the orthogonal decomposition with respect to the Killing form
on ${\rm Lie}(G_{\mathbb C})$.
\end{enumerate}
A $G$-Higgs bundle is then a pair $(P_{H_{\mathbb C}},\,\Phi)$, where $P_{H_{\mathbb C}}$ is a holomorphic principal 
$H_{\mathbb C}$--bundle on $\Sigma$
and $\Phi$ is a holomorphic section of $P_{H_{\mathbb C}}(\mathfrak{m}_{\mathbb C})\otimes 
K_{\Sigma}$ with $P_{H_{\mathbb C}}(\mathfrak{m}_{\mathbb C})$ being the holomorphic vector bundle associated to
$P_{H_{\mathbb C}}$ for the $H_{\mathbb C}$--module $\mathfrak{m}_{\mathbb C}$ given by the adjoint action.
In the case $G\,=\,{\rm SL}(2,\mathbb R)$, where $$H\,=\, {\rm SO}(2)\,\simeq\,{\rm U}(1),$$ we see that
$P_{H_{\mathbb C}}$ is a $\mathbb C^*$--bundle. In view of the
standard representation of $\mathbb C^*$, giving a $\mathbb C^*$--bundle is equivalent to giving a
holomorphic line bundle; let $L$ be the holomorphic line bundle corresponding to $P_{H_{\mathbb C}}$.
Then $$P_{H_{\mathbb C}}(\mathfrak{m}_{\mathbb C})\,=\,L^2\oplus L^{-2},$$ so the Higgs field $\Phi$ has two components,
namely $\alpha\,\in\, H^0(\Sigma,\, L^2\otimes K_{\Sigma})$ and $\delta\,\in\, H^0(\Sigma,\,L^{-2}\otimes K_{\Sigma})$.
We thus arrive at the following definition:

\begin{definition}
An ${\rm SL}(2,\mathbb R)$--{\it Higgs} bundle on $\Sigma$ is a triple $(L,\, (\alpha,\,\delta))$ where $L$ is a
holomorphic 
line bundle on $\Sigma$ while $\alpha$ and $\delta$ are holomorphic sections of $L^2\otimes K_{\Sigma}$ and 
$L^{-2}\otimes K_{\Sigma}$ respectively.
\end{definition}

\subsection{Moduli spaces}

There are good notions of semistability and stability
for $G$-Higgs bundles such that the $S$--equivalence classes of the semistable 
objects form an algebraic moduli space. Let $\mathcal{M}({\rm SL}(2,\mathbb R))$ be the moduli space of semistable
${\rm SL}(2,\mathbb R)$--Higgs bundles.
The inclusion map $${\rm SO}(2,\mathbb C)\,=\, {\mathbb C}^*\, \hookrightarrow\, {\rm SL}(2,\mathbb C)$$
in \eqref{el1} induces a finite map of moduli spaces (see Definition \ref{def1}) given by
\begin{align}
\mathcal{M}({\rm SL}(2,\mathbb R))&\,\longrightarrow\,\mathcal{M}({\rm SL}(2,\mathbb C))\\
[L,\, (\alpha,\,\delta)]&\,\longmapsto\, \left[L\oplus L^{-1},\;
\begin{pmatrix}0& \alpha\\ \delta&0\end{pmatrix}\right]\nonumber
\end{align}
The non-abelian Hodge correspondence (explained further in Section \ref{equations}) identifies 
$$
\mathcal{M}({\rm SL}(2,\C))\,\simeq\, {\rm Rep}(\pi_1(\Sigma),\, {\rm SL}(2,\mathbb C))\, ,
$$
where ${\rm Rep}(\pi_1(\Sigma),\, {\rm SL}(2,\mathbb C))$ denotes conjugacy classes of reductive representations
of $\pi_1(\Sigma)$ in ${\rm SL}(2,\mathbb C)$. 
This correspondence identifies
$$
\mathcal{M}({\rm SL}(2,\mathbb R))\,\simeq\, {\rm Rep}(\pi_1(\Sigma),\, {\rm SL}(2,\mathbb R))\, .
$$
While $\mathcal{M}({\rm SL}(2,\C))$ is connected, the moduli space
$\mathcal{M}({\rm SL}(2,\mathbb R))$ has multiple 
connected components, first enumerated for ${\rm Rep}(\pi_1(\Sigma),\,{\rm SL}(2,\mathbb R))$ by Goldman, \cite{Go},
and for $\mathcal{M}({\rm SL}(2,\mathbb R))$ by Hitchin \cite{Hi}. In the parlance of Higgs bundles, the
components are labeled by the degree 
of the line bundle $L$, subject to the Milnor-Wood inequality
\begin{equation}\label{MW}
|{\rm degree}(L)|\,\le\, g-1\, .
\end{equation}
The components with $|{\rm degree}(L)|\,< \,g-1$ are connected but there are $2^{2g}$ connected components with
$|{\rm degree}(L)|\,=\,g-1$. They correspond to the theta characteristics of $\Sigma$.

In the case ${\rm degree}(L) \,= \,g-1$, the semistability condition requires the component of the Higgs field in
$L^{-2}\otimes K_{\Sigma}$, i.e., $\delta$, to be non-trivial. This forces the ${\rm SL}(2,\mathbb R)$-Higgs bundles in
these components to be of the form $(K^{1/2}_{\Sigma},\, (\alpha,\, 1))$ where $K^{1/2}_{\Sigma}$ is a theta
characteristic, and $1$ denotes the unit constant section in $(K^{1/2}_{\Sigma})^{-2}\otimes K_{\Sigma}\,\simeq\,
\mathcal{O}_{\Sigma}$. The corresponding stable $\text{SL}(2,{\mathbb C})$--Higgs bundles are
of the form
\[\left(K^{1/2}_\Sigma\oplus K^{-1/2}_\Sigma,\;
\begin{pmatrix}0& \alpha\\ 1&0\end{pmatrix}\right)\, \]
where $\alpha\, \in\, H^0(\Sigma,\,K^{\otimes 2}_\Sigma)$ \cite{Hi}.
For each choice of $K^{1/2}_{\Sigma}$, the component in $\mathcal{M}({\rm SL}(2,\mathbb R))$ is thus
parameterized by the quadratic differentials $\alpha$, thereby identifying the component with
$H^0(\Sigma, \, K_{\Sigma}^2)\,\simeq\,{\mathbb C}^{3g-3}$.

In the non-maximal components, i.e., those for which $|{\rm degree}(L)|\,<\, g-1$, the stability condition 
still requires $\delta\,\ne\, 0$, but this condition no longer restricts the bundle $L^{-2}\otimes K_{\Sigma}$
to just one point in the appropriate Picard variety. The components are thus no longer topologically 
trivial spaces. In fact, they have the structure of a vector bundle over a symmetric product of the curve 
$\Sigma$ (see \cite{Hi}).

\subsection{Equations and correspondences}\label{equations}

The stability of $(L,\,(\alpha,\,\delta))$ is equivalent to the existence of a hermitian metric on $L$, say $l$, 
satisfying the Yang-Mills equation. Given a fixed metric on $\Sigma$ with K\"ahler form 
$\omega$, the Yang-Mills equation is
\begin{equation}\label{SL2Reqtn}
\sqrt{-1}\Lambda_{\omega}F_l+|\alpha|_l^2-|\delta|^2_l\,=\,0\, ,
\end{equation}
where $\Lambda_{\omega}$ denotes contraction with the K\"ahler form, $F_l $ denotes the curvature of the Chern connection determined by $l$, 
and the norms are with respect to the metric 
induced on $L^{\pm 2}\otimes K_{\Sigma}$ by $l$ and the metric on $\Sigma$. With respect to local holomorphic frames for 
$L$ and $K_{\Sigma}$, the metric is given by smooth positive function (also denoted by $l$) and the equation in
\eqref{SL2Reqtn} becomes
\begin{equation}\label{SL2Reqtnlocal}
\sqrt{-1}\Lambda_{\omega}(\partial(l^{-1}\overline{\partial}l))-l^2|\alpha|^2+l^{-2}|\delta|^2\,=\,0\, .
\end{equation}
The corresponding Yang-Mills equation for a metric $h$ on a stable ${\rm SL}(2,\mathbb C)$-Higgs bundle
$(E,\,\Phi)$ is
\begin{equation}\label{SL2Ceqtn}
\Lambda_{\omega}(F_h+[\Phi,\, \Phi^{*_h}])\,=\,0\, . 
\end{equation}
If $(E,\,\Phi)$ is of the form \eqref{SL2CHigss} then the solution has the form 
$h\,=\,\begin{bmatrix}l&0\\0&l^{-1}\end{bmatrix}$, so that with respect to a local holomorphic frame
$$\Phi^{*_h}\,=\,\begin{pmatrix}0& l^{-2}\overline{\delta}\\ l^2\overline{\alpha}&0\end{pmatrix}$$
and \eqref{SL2Ceqtn} reduces to \eqref{SL2Reqtnlocal}.

Denoting the Chern connection for $h$ by $\nabla_h$, the condition \eqref{SL2Ceqtn} implies that the connection
\begin{equation}\label{nabla+phi}
\nabla_h+\Phi+\Phi^{*_h}
\end{equation}
is flat, and hence it defines a representation of $\pi_1(\Sigma)$ in ${\rm SL}(2,\mathbb C)$. This leads to one
direction of the non-Abelian Hodge correspondence. To explain the other direction in the case of
${\rm SL}(2,\mathbb C)$-Higgs bundles, let $\mathbb E$ be a trivial determinant rank two bundle with a flat
structure defined by flat connection, say $\nabla_0$. Any hermitian metric on $\mathbb E$, say $h$, then
decomposes $\nabla_0$ as
\begin{equation}
\nabla_0\,=\,\nabla_h+\Phi+\Phi^{*_h},
\end{equation}
where the notation is as in \eqref{nabla+phi}. We thus get a Higgs bundle $(E,\,\Phi)$, where $E$ denotes the
holomorphic bundle defined by $\nabla_h^{0,1}$ on $\mathbb E$. Using the flat structure we can identify
\begin {equation}
E\,=\,\widetilde{\Sigma}\times_{\rho}\mathbb C^2,
\end{equation}
where $\widetilde{\Sigma}$ is the universal cover and $\rho$ denotes the holonomy representation defined by
$\nabla_0$. It follows that a hermitian metric on $E$ is equivalent to a $\rho$-equivariant map 
\begin{equation}
f\,:\,\widetilde{\Sigma}\,\longrightarrow\, {\mathbb H}^2\,=\, {\rm SL}(2,\mathbb C)/{\rm U}(2)\ .
\end{equation}
Moreover, if $f$ is harmonic then the metric satisfies the equation in \eqref{SL2Ceqtn}.

\subsection{Relation to hyperbolic structures}

If $L\,=\,K_{\Sigma}^{1/2}$, the solutions of the Higgs bundle equation for $(K_{\Sigma}^{1/2},\, (\alpha,\,1))$
are metrics 
on $K^{1/2}_{\Sigma}$. Each of these determines a positive-definite symmetric 2-tensor on the underlying surface 
which in local holomorphic coordinates is given by
$$
\alpha dz\otimes dz+(l^{-2}+l^2|\alpha|^2)dz\otimes d\overline{z}+\overline{\alpha}d\overline{z}\otimes d\overline{z}\, .
$$
Here by abuse of notation we write $\alpha\,=\,\alpha dz^2$. The non-degeneracy of these 2-tensors implies that they 
determine Riemannian metrics. If $\alpha\,=\,0$ this is the uniformizing hyperbolic structure on $\Sigma$. For the 
other Higgs bundles of the form $(K_{\Sigma}^{1/2},\, (\alpha,\,1))$, meaning for those with $\alpha\,\ne\,0$, the 
metrics are other hyperbolic structures on the surface underlying $\Sigma$. This identifies the Teichm\"uller space with 
a maximal component of $\mathcal{M}({\rm SL}(2,\mathbb R))$ and hence with $H^0(\Sigma,\, K_{\Sigma}^2)$. This 
recovers the fact that, fixing this given base point, the Hopf-type differentials $\alpha\, \in\, 
H^0(\Sigma,\,K^{\otimes 2}_\Sigma)$ parametrize the Fricke space of equivalence classes of marked hyperbolic metrics 
on $\Sigma$ \cite{Wo} (which by uniformization theorem identifies with the Teichm\"uller space ${\mathcal 
T}_{g_\Sigma}$ \cite{St}).

Inspired by the previous Hitchin's results in \cite{Hi}, the aim here is to give a similar description of 
uniformization for branched Riemann surfaces from the point of view of Higgs bundles. To that end we consider ${\rm 
SL}(2,\mathbb R)$-Higgs bundles with ${\rm degree}(L)\,<\,g-1$. In this case the solutions of \eqref{SL2Reqtn} 
define a non-negative symmetric 2-tensor on $\Sigma$ locally given by
\begin{equation}\label{degenhypstructure}
\alpha\delta dz\otimes dz+(l^{-2}|\delta|^2+l^2|\alpha|^2)dz\otimes
d\overline{z}+\overline{\alpha}\overline{\delta}d\overline{z}\otimes d\overline{z}\, .
\end{equation}

\subsection{Harmonic maps to hyperbolic space and Higgs bundles}\label{harmaps}

The space of positive definite hermitian $2\times 2$ matrices of determinant 1
\[{\mathcal H}\, :=\, \{A\,\in\, \mathrm{SL}(2,\C)\,\mid\, \overline{A}^t\,=\,A,\ A \, >\, 0\}\]
is identified with the symmetric space $\mathrm{SL}(2,\C)/\mathrm{SU}(2)$ using the polar decomposition
\[\mathrm{SL}(2,\C)\,=\, {\mathcal H}\mathrm{SU}(2)\]
of $\mathrm{SL}(2,\C)$. The corresponding Riemannian metric is given by the quadratic form
\[q(Y)\,:=\, -2\det(Y)\]
on $\mathfrak{gl}(2,\C)$, which produces a Riemannian metric on $\mathcal H$. 
Note that $-2\text{det}(Y)\,=\, \text{tr}(Y^2)$ for 
any trace-free hermitian matrix $Y\,\in\, \mathfrak{gl}(2,\C)$.

Let $\Sigma$ be a compact connected Riemann surface of genus $g_\Sigma$, with
$g_\Sigma\, \geq\, 2$. Fix a base point $x_0\, \in\, \Sigma$. Let
$$
p\, :\, \widetilde{\Sigma}\, \longrightarrow\, \Sigma
$$
be the corresponding universal cover.

Take an irreducible representation
\begin{equation}\label{e1}
\rho\, :\, \pi_1(\Sigma,\, x_0)\, \longrightarrow\,\mathrm{SL}(2,\C)\, .
\end{equation}
Let
$$f\,:\,\widetilde{\Sigma}\, \longrightarrow\, {\mathcal H}$$ 
be the corresponding $\pi_1(\Sigma,\, x_0)$--equivariant harmonic map \cite{Co}. We describe below the geometric construction of the harmonic map $f$.

Consider the $C^{\infty}$ trivial rank two vector bundle
$$\underline{\C}^2\, :=\, \Sigma\times \C^2\, \longrightarrow\,\Sigma$$
equipped with the standard (constant) hermitian structure $h_0$. Let
$(\overline{\partial}^0,\, \Phi)$ be a Higgs bundle structure on $\underline{\C}^2$
corresponding to $\rho$; so $\overline{\partial}^0$ is a Dolbeault operator on
$\underline{\C}^2$, and $\Phi$ is a Higgs field on the holomorphic vector bundle
$(\underline{\C}^2,\, \overline{\partial}^0)$, such that the Higgs bundle
$(\underline{\C}^2,\, \overline{\partial}^0,\, \Phi)$ corresponds to $\rho$. This means that 
$(\underline{\C}^2,\, \overline{\partial}^0,\, \Phi, \, h_0)$ solves \eqref{SL2Ceqtn}
such that the representation $\rho$ is the monodromy representation of the flat connection
$\nabla+\Phi+\Phi^{* _{h_0}},$ where 
 $\nabla$ is the Chern connection on the
holomorphic hermitian bundle $(\underline{\C}^2,\, \overline{\partial}^0, \, h_0)$.
 Since
$\rho$ is irreducible, $(\overline{\partial}^0,\, \Phi)$ is uniquely
determined by $\rho$.

For all $\lambda\, \in\, \C^*$, let
\begin{equation}\label{eq:associatedfamily}
\nabla^\lambda\,=\, \nabla+\lambda^{-1}\Phi +\lambda\Phi^* 
\,=\, d+\xi^\lambda=d+\lambda^{-1}\xi_{-1}+\xi_0+\lambda\xi_1
\end{equation}
be the family of flat ${\rm SL}(2,\C)$--connections; here
$$\xi_{-1}\,\in\,\Omega^{(1,0)}(\Sigma,\mathfrak{sl}(2,\C)),\ \
\xi_{1}\,\in\,\Omega^{(0,1)}(\Sigma,\mathfrak{sl}(2,\C)),\ \
\xi_{0}\,\in\,\Omega^1(\Sigma,\mathfrak{su}(2,\C))$$
with
\[
d\xi^\lambda+\xi^\lambda\wedge\xi^\lambda\,=\,0\, 
\]
and
\[\overline{\partial}^0=d''+\xi_0^{0,1}.\]
The above family $\xi^\lambda$ satisfies the following reality condition:
\begin{equation}\label{e2}
\overline{\xi_{-1}}^t\,=\,\xi_1\quad \text{ and } \quad \overline{\xi_0}^t\,=\, -\xi_0\, .
\end{equation}
From \eqref{e2} it follows that
\begin{equation}\label{eq:omegareal}
-\xi^{-1/\overline{\lambda}} \,= \,\overline{\xi^\lambda}^t\, . 
\end{equation}
Consider a parallel frame on the universal covering $\widetilde{\Sigma}
\,\longrightarrow\,\Sigma$
\begin{equation}\label{pf}
{\widehat F}\,\colon\,\C^*\times\widetilde{\Sigma}
\,\longrightarrow\,{\rm SL}(2,\C)
\end{equation}
for the pull-back of $(\underline{\C}^2,\,\nabla^\lambda)$ on $\widetilde{\Sigma}$.
With the notation ${\widehat F}^\lambda\,=\, {\widehat F} (\lambda, \cdot)$, we have 
\begin{equation}\label{eq:Flambdaparallel}
d{\widehat F}^\lambda\,=\,-\xi^\lambda {\widehat F}^\lambda\, .
\end{equation}
In particular, we have
\begin{equation}\label{diffframeeq}
dF\,=\, -\xi^1 F\,\quad\,\text{ and } \quad\, d(\overline{F}^t)\,=\,\overline{F}^t \xi^{-1}\, ,
\end{equation}
where $\xi^1$ and $\xi^{-1}$ are $\xi^{\lambda}$ for $\lambda\,=\,1$ and
$\lambda\,=\,-1$ respectively, and 
\begin{equation}\label{pf2}
F\,=\,\widehat{F}^1
\end{equation}
(see \eqref{eq:Flambdaparallel}). Notice that the second equation in \eqref{diffframeeq} comes from the fact that $\overline{\widehat{F}^{\lambda}}^t$ is a parallel frame for the flat connection
$\overline{{\nabla}^{\lambda}}^t$, equivalently $\overline{{\nabla}^{\lambda}}^t\overline{\widehat{F}^{\lambda}}^t \, = \, 0.$

We call $\widehat F$ the \textit{extended frame}.

\begin{lemma}\label{theharmonicmap}
Let \[((d+\xi_0)'',\, \xi_{-1})\] be a Higgs pair 
solving the Higgs bundle equation \eqref{SL2Ceqtn}
for the trivial hermitian structure $h_0$.
Then, the $\rho$ equivariant harmonic map corresponding to this Higgs pair is
\[f\, =\, \overline{F}^t F\,\colon\,\widetilde{\Sigma}\,\longrightarrow\, {\mathcal H}\, ,\]
where $F$ is the map in \eqref{pf2}.
\end{lemma}

\begin{proof}
It is a straightforward computation to show that the map $f\, =\, \overline{F}^t F$ is harmonic. Indeed, with the notations as in (\ref{nabla+phi}), the differential of $f$ coincides with the pull-back to $\widetilde{\Sigma}$ of $\Phi+\Phi^{* _{h_0}}$, 
where $f^*(T\mathcal H)$ is identified with the pull-back of $\text{Sym}_{h_0}(\underline{\C}^2)$. Moreover, the latter identification intertwines the pull-back through $f$ of the Levi-Civita connection of $\mathcal H$ with the pull-back
of the Chern connection $\nabla_{h_0}$ in (\ref{nabla+phi}). Equation (\ref{SL2Ceqtn}) implies $$d_{\nabla_{h_0}}^*(df)\,= \,d_{\nabla_{h_0}}^*(\Phi+\Phi^{* _{h_0}})\,=0$$
which characterizes the fact that $df$ (and hence $f$) is harmonic.

The map $f$ is equivariant
with respect to the $\mathrm{SL}(2,\C)$ representation $\rho$ of $\pi_1(\Sigma,\, x_0)$ and the 
action 
\[ A\cdot g\,:=\,\overline{g}^t A g\]
of $g\,\in\,\mathrm{SL}(2,\C)$ on $A\,\in\,\mathcal H.$ The lemma now
follows from the uniqueness of the equivariant harmonic map corresponding to a Higgs bundle.
\end{proof}

\begin{remark}\label{remgauge}
Note that for a unitary gauge $g\,\colon\,\Sigma\,\longrightarrow\,\mathrm{SU}(2)$ and
\[\widetilde{\nabla}^\lambda\,=\, g\circ \nabla^\lambda\circ g^{-1}\]
we have the extended frame
\[ \widetilde{F}\,=\,g{\widehat F}\, .\]
This yields the same harmonic map
\[\overline{\widetilde{F}^{1}}^t \widetilde{F}^{1}\,=\,\overline{F}^t\overline{g}^t g F
\,=\,\overline{F}^t F\, .\]
\end{remark}

\section{Conical constant negative curvature metrics via Higgs bundles}

Let $D$ be an effective divisor on $\Sigma$ of even degree. We have $ {\mathcal O}_\Sigma\, 
\subset\, {\mathcal O}_\Sigma(D)$, and the canonical section of ${\mathcal O}_\Sigma(D)$ given by the 
constant function $1$ on $\Sigma$ will be denoted by $\delta$. So the divisor for $\delta$ is $D$.

Let $L$ be a holomorphic line bundle
on $\Sigma$ such that
\begin{equation}\label{g1}
L^{\otimes 2}\,=\, K_\Sigma\otimes {\mathcal O}_\Sigma(-D)\, .
\end{equation}
Our ${\rm SL}(2,\mathbb R)$-Higgs bundles are of the form
\begin{equation}
(L,\,(\alpha,\,\delta))\, ,
\end{equation}
where $\alpha\,\in\, H^0(\Sigma,\,K^{\otimes 2}_\Sigma\otimes {\mathcal O}_\Sigma(-D))$, and
\[\delta\,\in\, H^0(\Sigma, \,K_\Sigma\otimes (K_\Sigma\otimes {\mathcal O}_\Sigma(-D))^*)
\,=\,H^0(\Sigma,\, {\mathcal O}_\Sigma(D))\]
is defined above. Viewed as an ${\rm SL}(2,\mathbb C)$-Higgs bundle this has the form
\[E\,=\,L\oplus L^*,\,\ \Phi\,=\, \begin{pmatrix}0&\alpha\\ \delta &0\end{pmatrix}\, .\]

In order to analyze the behavior at the zeros of $\delta$, i.e., at the support of $D$, we consider 
a local gauge $g$ such that $\xi_0$ is diagonal and $\xi_{\pm 1}$ are off-diagonal.

With respect to a local holomorphic coordinate $z$ on $U\,\subset\, \Sigma$ 
the connection 1-forms are given (in the local gauge $\underline \C^2\,=\,L\oplus L^\perp$) as follows:
\begin{equation}\label{localnormalform}
\begin{split}
\Phi&\,=\,\xi_{-1}\,=\,\begin{pmatrix} 0 & a \\ d &0\end{pmatrix} dz\\
\Phi^*&\,=\,\xi_{1}\,=\,\begin{pmatrix} 0 &\overline{d}\\ \overline{a} &0\end{pmatrix}d\overline{z}\\
\xi_0&\,=\,\begin{pmatrix} c &0\\ 0&-c\end{pmatrix}dz-\begin{pmatrix} \overline{c} &0 \\ 0 &-\overline{c}\end{pmatrix}d
\overline{z}\\
\end{split}
\end{equation}
for functions $a,\,d,\, c\,\colon\, U\,\longrightarrow\, \C$. We note that these are not holomorphic
functions as we have a conjugation by a unitary gauge: for instance, if $l$ is a real function defining the metric on $L$ given by \eqref{SL2Reqtnlocal}\ , then the expression of $\Phi^*$ in $(\ref{localnormalform})$ is 
$\begin{pmatrix} 0 & l^{-2} \overline{d}\\ l^2\overline{a} &0\end{pmatrix}d\overline{z}$.

 However, the functions $a, \, d, \, c$ satisfy a $\overline{\partial}$-equation,
in particular, their order of vanishing at points is well-defined. Also note 
that $\alpha$ and $\delta$ are locally determined by $a$ and $d$.

\begin{proposition}\label{h2subseth3}
The equivariant harmonic map $f$ associated to a ${\rm SL}(2,\mathbb R)$-Higgs bundle takes values in a hyperbolic 2-space.
In particular, the induced symmetric 
bilinear form $\text{tr}((\Phi_p+\Phi_p^*)^2)$ of the map $f$ yields the constant negative curvature metric provided it 
does not degenerate.
\end{proposition}

\begin{proof}
Consider the associated family $d+\xi^\lambda$ of flat connections \eqref{eq:associatedfamily} and the extended frame
$\widehat F$ on an ${\rm SL}(2,\mathbb R)$-Higgs bundle as above. For $F
\,=\,\widehat{F}^{1}$, Lemma \ref{theharmonicmap} proves that the map 
$f\, =\, \overline{F}^t F\,\colon\,\widetilde{\Sigma}\,\longrightarrow\, {\mathcal H}$ to the $3$-hyperbolic space $\mathcal H$ is harmonic.

Here the connections $d+\xi^{1}$ and $d+\xi^{-1}$ are gauge equivalent by the gauge 
transformation $T$ which is diagonal with respect to $L\oplus L^\perp$ with eigenvalues $\pm \sqrt{-1}$. 

We obtain that $(\overline{F}^t)^{-1}\,=\,TFT^{-1}$; compare it with \eqref{diffframeeq}. A direct computation shows that 
the harmonic map $f\,=\,{\overline F}^tF$ takes values in the two dimensional submanifold of $\mathcal H$ formed by 
hermitian matrices with equal diagonal entries. When endowed with the restriction of the hyperbolic metric of 
$\mathcal H$, this two-dimensional manifold is a totally geodesic hyperbolic plane in $\mathcal H$.

Remark \ref{remgauge} shows that the image of the map $f$ is invariant by the unitary gauge transformation $T$ and 
terminates the proof.
\end{proof}

Proposition \ref{h2subseth3} and its proof are analogous to \cite[Proposition 1.9]{Hi2} for harmonic 
maps into a totally geodesic 2-sphere $S^2\,\subset\, S^3$ inside the 3-sphere. 

Consider the symmetric $2$-tensor
\begin{equation}\label{eq-conf-class}
\begin{split}
\text{tr}((\Phi_x+\Phi_x^*)^2)&\,=\,
ad\ dz\otimes dz+(l^{-2}|d|^2+l^2|a|^2)dz\otimes d\overline{z}+\overline{a}\overline{d}d\overline{z}\otimes
d\overline{z}\\
&=\,l^{-2}( d\, dz+l^2 \overline{a} \, d\overline{z})(\overline{d}\, d\overline{z}+l^2 a \, dz)\, ,
\end{split}
\end{equation}
where $l$ is a real function defining the metric on $L$ given by \eqref{SL2Reqtnlocal}
while $a$ and $d$ are as in \eqref{localnormalform}. A short computation shows (compare also with the proof
of \cite[Theorem 11.2]{Hi}) that the symmetric bilinear form in \eqref{eq-conf-class} is positive definite at a 
point $x\,\in\, U$ if and only if
\begin{equation}\label{eq-non-deg-met}
l^{-2}(x)|d(x)|^2\,\neq\, l^2(x)|a(x)|^2\, .
\end{equation}

\begin{lemma}\label{lem:metdeg}
Assume that there exist points $p,\,q\,\in\, \Sigma$ such that
\[{\rm ord}_p\delta\,<\,{\rm ord}_p\alpha\,\quad { and }\quad\, {\rm ord}_q\alpha
\,<\, {\rm ord}_q\delta\, ,\]
where ${\rm ord}$ denotes the order of vanishing.
Then, the induced bilinear form ${\rm tr}((\Phi_p+\Phi_p^*)^2)$ fails to be positive
definite somewhere in $\Sigma\setminus{\rm supp}(D)$.
\end{lemma}

\begin{proof}
There exists an open and connected subset $U\,\subset\, \Sigma$ with $p,\,q\,\in\, U$, and a holomorphic
coordinate $z\,\colon\, U\,\longrightarrow\, \mathbb C$, and a unitary gauge $g$
on $U$ such that the connection takes the form given in \eqref{localnormalform}.
By continuity, 
\[l^{-2}(x)|d(x)|^2\,\neq\, l^2(x)|a(x)|^2\]
cannot hold for all $x\,\in\, U$.
\end{proof}

\begin{remark}
Let $D$ be an effective divisor on $\Sigma$ of even degree, and let $(L,\,(\alpha,\,\delta))$ be an
${\rm SL}(2,\mathbb R)$-Higgs bundle, where $\delta$ is the canonical section of $\mathcal O_\Sigma(D)$
given by the constant function $1$ on $\Sigma$,
such that \eqref{g1} holds. Assume that $0 \, \leq \text{degree}(L) \, \leq g-1$. If the induced 
symmetric bilinear form is nondegenerate away from $D$, then we have ${\rm Div}(\alpha)\,\geq\, D$
due to Lemma \ref{lem:metdeg}. In particular, for the construction of constant curvature $-4$ metrics with conical 
singularities at $\text{supp}(D)$ we have to impose the following condition:
\[\alpha\,\in \,H^0(\Sigma,\,K_\Sigma^{\otimes 2}\otimes{\mathcal O}_\Sigma(-2D))\,\subset\,
H^0(\Sigma,\,K_\Sigma^{\otimes 2}\otimes{\mathcal O}_\Sigma(-D)).\]
\end{remark}

\begin{theorem}\label{thm1}
Let $D\,=\,\sum_i n_ix_i$ be an effective divisor on a compact Riemann surface $\Sigma$,
with $g_\Sigma\, \geq\, 2$, such that
\[N\,:=\,\sum_i n_i\,<\,2g_\Sigma-2\, ,\]
and $N$ is an even integer. Let $L$ be a holomorphic line bundle on $\Sigma$ such that
\begin{equation}\label{L}
K_\Sigma\otimes L^{-2}\,=\,{\mathcal O}_\Sigma(D)\, .
\end{equation}

\noindent Then, there is a unique largest connected subset
\begin{equation}\label{Ucond}
{\mathcal U}\, \subset\, H^0(\Sigma,\,K_\Sigma^{\otimes 2}\otimes{\mathcal O}_\Sigma(-2D))
\end{equation}
containing $0$ such that for all $\alpha\, \in\, {\mathcal U}$,
the harmonic map associated to the Higgs pair
\[\left(L\oplus L^{-1},\;
\begin{pmatrix}0&\alpha\\ \delta&0\end{pmatrix}\right)\]
gives a constant negative curvature $-4$ metric with conical singularities of order $n_i$ at $x_i.$

The above subset ${\mathcal U}$ is either $\{0\}$ or is open and satisfies the condition that
$$
H^0(\Sigma,\,K_\Sigma^{\otimes 2}\otimes{\mathcal O}_\Sigma(-2D-D_{\rm red}))
\, \subset\, {\mathcal U} \, \subset\,
H^0(\Sigma,\,K_\Sigma^{\otimes 2}\otimes{\mathcal O}_\Sigma(-2D))\, ,
$$
where $D_{\rm red}\,=\,\sum_i x_i$ is the reduced divisor.

When $\alpha\,=\,0$ the corresponding metric is in the conformal class of $\Sigma$.
\end{theorem}

\begin{proof}
Note that for a generic point $p\,\in\, \Sigma$, the induced bilinear form 
$\text{tr}((\Phi_p+\Phi_p^*)^2)$ is positive definite because $\delta$ and $\alpha$ are
not constant multiplies of each other by the assumption that $N\,<\,2g_\Sigma-2.$

For $\alpha\,=\,0$, consider the --- possibly degenerate --- metric on
$\Sigma$ obtained by pulling back the canonical metric of the hyperbolic
2-space using the equivariant harmonic map $f$ in Proposition \ref{h2subseth3}. It is 
conformal and it degenerates only at the zeros of the Higgs field. These are exactly 
the zeros of the section $\delta$, and we obtain a conical metric of constant curvature $-4$ in the 
conformal class of $\Sigma$. The cone angle at $x_i$ is $2\pi n_i$.

Next, we assume that
\begin{equation}\label{order}
0\, \not=\, \alpha\,\in\, H^0(\Sigma,\,K_\Sigma^{\otimes 2}\otimes{\mathcal O}_\Sigma(-2D-
D_{\rm red}))\, .
\end{equation}
With respect to the local form given in \eqref{localnormalform}, we have
\[\lim_{x\to x_i}l^2(x)\frac{|a(x)| }{|d(x)| }\,=\,0\]
because \eqref{order} holds. 
Hence, we can apply the strategy of the proof of 
\cite[Theorem 11.2(i)]{Hi}: the absolute maximum of the function
\[x\,\longmapsto \,l^4(x)\frac{|a(x)|^2}{|d(x)|^2}\]
does exist at a point $p\,\in\, \Sigma\setminus\{x_1,\,\cdots,\,x_m\}$, and the strong maximum
principle gives that
\[l^4(x)\frac{|a(x)|^2}{|d(x)|^2}\,<\,1\, .\]
By \eqref{eq-non-deg-met}, the induced symmetric bilinear form
\[\text{tr}((\Phi+\Phi^*)^2)\]
is a Riemannian metric away from the points $x_i$. 

Finally, assume that
\begin{equation}\label{orderineq}
\alpha\,\in\, H^0(\Sigma,\,K_\Sigma^{\otimes 2}\otimes{\mathcal O}_\Sigma(-2D))\, .
\end{equation}
We note that the solution of the self-duality equation depends continuously on the initial
data $\alpha$, because we have fixed $L$ and $\delta$. Also,
\[l^4(x)\frac{|a(x)|^2}{|d(x)|^2}\]
is bounded near $x_i$ by \eqref{orderineq}. 
Hence, for any $\alpha\,\in\,\mathcal U\, \subset\,H^0(\Sigma,\,K_\Sigma^{\otimes 2}\otimes{\mathcal O}_\Sigma(-2D))$
in an appropriate open neighborhood $\mathcal U$ of $0$,
the induced symmetric bilinear form
\[\text{tr}((\Phi+\Phi^*)^2)\]
is non-degenerate by continuity, and it defines a constant curvature -4 metric.

It remains to show that in the case $\alpha\,\neq\, 0$ the singularity $x_i$ is conical of order $n_i,$
i.e., that there exist polar coordinates $(r,\,\theta)$ centered at $x_i$ such that the metric of constant curvature $-4$ is given by
\[\tfrac{1}{4}((dr)^2+(n_i+1)^2\sinh(r)^2 (d\theta)^2).\]
To this end consider the measurable Beltrami differential $\beta$ locally given by
\[l^2 \frac{\overline a}{d} \frac{d\overline z}{dz}\, ;\]
compare with \eqref{eq-conf-class}. From the above it follows that for $\alpha\,\in\,\mathcal U$, we have
\[\parallel \beta\parallel_\infty\,<\,1\, ,\]
and by the measurable Riemann mapping theorem (see \cite{AB}) there does exist a Riemann surface structure $M$ on 
the compact topological surface underlying $\Sigma$ such that the metric $\text{tr}((\Phi+\Phi^*)^2)$ on 
$M\setminus\text{supp}(D)\,=\,\Sigma\setminus\text{supp}(D)$ is in the conformal class determined by $M$. Then, the 
(equivariant) harmonic map $f$ is given by a holomorphic map $\widetilde{f}\,\colon\, \widetilde{M}\,
\longrightarrow\, \mathbb D$ defined on the 
universal covering $\widetilde M\,\longrightarrow\, M$ to the unit disc $\mathbb D$.
It can be shown that this holomorphic map is branched of order $n_i$ at the preimage $\widetilde{x}_i
\,\in\,\widetilde M$ of $x_i\,\in\,\text{supp}(D)$. Denote the branch order of $\widetilde f$
at $\widetilde{x}_i$ by $k.$ Clearly, the constant curvature metric must then have a conical singularity of order
$k$ at $x_i\in M$. 
Note that the order of a conical singularity is determined (up to integers) by the monodromy along simple closed 
curves around the singularity and the conical order at $x_i$ depends continuously on $\alpha\,\in\,\mathcal U$. As 
$k\,=\,n_i$ for $\alpha\,=\,0$, the result follows.
\end{proof}

\begin{remark}
In general, the Beltrami differential $\beta$ in the proof of Theorem \ref{thm1} is not smooth at the support of
$D$. Thus, the Riemann surface structure $M$ yields a holomorphic atlas on the surface underlying $\Sigma$ which 
is, in general, not compatible with the holomorphic atlas defining the Riemann surface $\Sigma$.

Notice that by \eqref{L} we have $0\,\le\, \text{degree}(L)\,\le\, g-2$. The $\mathrm{SL}(2,\mathbb{R})$- Higgs bundles 
defined by the data $(L,(\alpha,\delta))$ thus determine points in the non-maximal components of the moduli space 
$\mathcal{M}(\mathrm{SL}(2,\mathbb{R}))$, i.e., in the connected components where the degree of the line bundle is 
strictly less than the Milnor-Wood bound in \eqref{MW}. As described by Hitchin in \cite[Proposition 10.2 (ii)]{Hi} 
these components are $2^{2g}$-fold covers of vector bundles over $Sym^N\Sigma$, the symmetric products of the curve 
$\Sigma$. The points in $Sym^N\Sigma$ correspond to the divisor $D$ in Theorem \ref{thm1} and the subsets $\mathcal{U}$ 
are in the fibers of the vector bundle. The fibers in which $\mathcal{U}$ contain non-zero elements are determined by 
Brill-Noether considerations. In particular, $\mathcal{U}$ is open whenever $\text{degree}(K_\Sigma^{\otimes 
2}\otimes{\mathcal O}_\Sigma(-2D-D_{\rm red}))\,>\,g-1$.
\end{remark}

\section{The case of odd degree divisors}\label{sec: odd degree}

As before, $\Sigma$ is a compact connected Riemann surface with
$g_\Sigma\, :=\, \text{genus}(\Sigma) \, \geq\,2$. 
Let $D\,=\,\sum_i n_ix_i$ be an effective divisor on $\Sigma$
such that
\begin{enumerate}
\item $N\,:=\,\sum_i n_i\,<\,2g_\Sigma-2$, and

\item $N$ is an odd integer.
\end{enumerate}

Fix an unramified connected covering
\begin{equation}\label{vp}
\varphi\, :\, X\, \longrightarrow\, \Sigma
\end{equation}
of degree two. We have $g_X\, :=\, \text{genus}(X)\,=\, 2g_\Sigma-1$.
Define $\widetilde{D}\, :=\, \varphi^*D$. So
$$
\text{degree}(\widetilde{D})\,=\, 2\cdot\text{degree}(D)\,=\,2N\, .
$$
The canonical section of ${\mathcal O}_X(\widetilde{D})$ given by the constant function $1$
on $X$ will be denoted by $\delta_1$.

We note that
$$
\text{degree}(\widetilde{D})\,=\, 2N\, <\, 4g_\Sigma-4\,=\, 2g_X-2\, .
$$
Fix a holomorphic line bundle $\mathcal{L}$ on $X$ such that
\begin{equation}\label{cl}
\mathcal{L}^{\otimes 2}\,=\, K_X\otimes {\mathcal O}_X(-\widetilde{D})\, .
\end{equation}

Set $(\Sigma,\, D)$ in Theorem \ref{thm1} to be the above pair $(X,\, \widetilde{D})$.
We conclude the following:
\begin{enumerate}
\item There is a unique largest open subset
$$
{\mathcal U}'_X\, \subset\, H^0(X,\,K_X^{\otimes 2}\otimes{\mathcal O}_X(-2\widetilde{D}))
$$
containing $0$ such that for all $\alpha'\, \in\, {\mathcal U}'_X$,
the harmonic map associated to the Higgs pair
\begin{equation}\label{ha}
\left(\mathcal{L}\oplus \mathcal{L}^{-1},\;
\begin{pmatrix}0&\alpha'\\ \delta_1 &0\end{pmatrix}\right)
\end{equation}
gives a constant negative curvature $-4$ metric with conical singularities of order $2n_i$ at
each point of $\varphi^{-1}(x_i)$.

\item The above open subset ${\mathcal U}'_X$ satisfies the condition that
$$
H^0(X,\,K_X^{\otimes 2}\otimes{\mathcal O}_X(-2\widetilde{D} -\widetilde{D}_{\rm red}))
\, \subset\, {\mathcal U}'_X \, \subset\,
H^0(X,\,K_X^{\otimes 2}\otimes{\mathcal O}_X(-2\widetilde{D}))\, ,
$$
where $\widetilde{D}_{\rm red}\,=\,\varphi^{-1}(D_{\rm red})$ is the reduced divisor.

\item When $\alpha'\,=\,0$ the corresponding metric is in the conformal class of $X$.
\end{enumerate}

Since $\varphi$ is unramified, we have the inclusion map
$$
H^0(\Sigma,\,K_\Sigma^{\otimes 2}\otimes{\mathcal O}_\Sigma(-2D))\, \hookrightarrow\,
H^0(X,\,K_X^{\otimes 2}\otimes{\mathcal O}_X(-2\widetilde{D}))\, ,\ \ \omega\, \longmapsto\,
\varphi^*\omega\, ;
$$
using it $H^0(\Sigma,\,K_\Sigma^{\otimes 2}\otimes{\mathcal O}_\Sigma(-2D))$ will be considered as a
subspace of $H^0(X,\,K_X^{\otimes 2}\otimes{\mathcal O}_X(-2\widetilde{D}))$. Now define the open subset
\begin{equation}\label{cux}
\widehat{\mathcal U}\, :=\, {\mathcal U}'_X\cap
H^0(\Sigma,\,K_\Sigma^{\otimes 2}\otimes{\mathcal O}_\Sigma(-2D))\, \subset\,
H^0(\Sigma,\,K_\Sigma^{\otimes 2}\otimes{\mathcal O}_\Sigma(-2D))\, .
\end{equation}

If the holomorphic line bundle $\mathcal{L}$ is replaced by a different choice
$\mathcal{L}\otimes\xi$ of the square-root of $K_X\otimes {\mathcal O}_X(-\widetilde{D})$, then
$$
(\mathcal{L}\otimes\xi)^{-1}\,=\, \mathcal{L}^{-1}\otimes\xi^{-1}\,=\,
\mathcal{L}^{-1}\otimes\xi\, ,
$$
because $\xi$ is of order two. This implies that
$$
(\mathcal{L}\otimes\xi)\oplus (\mathcal{L}\otimes\xi)^{-1}\,=\,
(\mathcal{L}\oplus\mathcal{L}^{-1})\otimes\xi\, .
$$
Hence we have
\begin{equation}\label{pi}
{\mathbb P}((\mathcal{L}\otimes\xi)\oplus (\mathcal{L}\otimes\xi)^{-1})\,=\,
{\mathbb P}(\mathcal{L}\oplus\mathcal{L}^{-1})
\end{equation}
{}From this it follows immediately that the Galois group $\text{Gal}(\varphi)\,=\,
{\mathbb Z}/2{\mathbb Z}$ for $\varphi$ has a canonical lift to the projective bundle
$$
{\mathcal P}'\, :=\, {\mathbb P}(\mathcal{L}\oplus\mathcal{L}^{-1})\, \longrightarrow\, X\, .
$$
In other words, there is a holomorphic ${\mathbb C}{\mathbb P}^1$--bundle
\begin{equation}\label{pb}
\varpi\, :\, {\mathcal P} \, \longrightarrow\,\Sigma
\end{equation}
such that $\varphi^*{\mathcal P} \,=\, {\mathcal P}'$. 

{}From \eqref{pi} it follows that the holomorphic
$\text{PSL}(2,{\mathbb C})$--bundle ${\mathcal P}'$ on $X$ is independent
of the choice of the holomorphic line bundle $\mathcal L$ satisfying \eqref{cl}.
Consequently, the holomorphic $\text{PSL}(2,{\mathbb C})$--bundle ${\mathcal P}$ on
$\Sigma$ is independent of the choice of the holomorphic line bundle $\mathcal L$.

It should be clarified that the $C^\infty$ vector bundle $\mathcal{L}\oplus\mathcal{L}^{-1}$
is \textit{not} the pullback to $X$ of a $C^\infty$ vector bundle on $\Sigma$. Indeed, for any
$C^\infty$ vector bundle $V$ on $\Sigma$ such that ${\mathcal P}\,=\, {\mathbb P}(V)$,
the degree of $V$ is an odd integer. So $\text{degree}(\varphi^*V)\, \not=\, 0$,
implying that $\varphi^*V$ is not isomorphic to $\mathcal{L}\oplus\mathcal{L}^{-1}$.
An alternative way to see this is to note that the second Stiefel--Whitney class of
$\mathcal{L}\oplus\mathcal{L}^{-1}$ is nonzero, while the second Stiefel--Whitney class of
any vector bundle on $X$ which is pulled back from $\Sigma$ is zero.

The holomorphic projective bundle ${\mathcal P}\, \longrightarrow\, \Sigma$ can be explicitly described. For that
first note that
$$
{\mathbb P}(\mathcal{L}\oplus\mathcal{L}^{-1})\,=\, {\mathbb P}((\mathcal{L}\oplus\mathcal{L}^{-1})\otimes
\mathcal{L})\,=\, {\mathbb P}(\mathcal{L}^{\otimes 2}\oplus {\mathcal O}_X)
$$
because tensoring by a line bundle does not alter the projective bundle. We have
$$
\mathcal{L}^{\otimes 2}\,=\, K_X\otimes {\mathcal O}_X(-\widetilde{D})\,=\,
\varphi^*(K_\Sigma\otimes {\mathcal O}_\Sigma (-D))\, .
$$
These imply that
$${\mathcal P}\,=\, {\mathbb P}((K_\Sigma\otimes {\mathcal O}_\Sigma (-D))\oplus
{\mathcal O}_\Sigma)\, .
$$
In particular, the projective bundle ${\mathcal P}$ does not depend on the choice of the covering $\varphi$.

For any $\alpha'\, \in\, {\mathcal U}'_X$, we get a flat $\text{SL}(2,{\mathbb C})$--connection
on the $C^\infty$ vector bundle $\mathcal{L}\oplus\mathcal{L}^{-1}$ which is associated to the
Higgs bundle in \eqref{ha}; this flat $\text{SL}(2,{\mathbb C})$--connection will be denoted
by ${\mathcal D}'(\alpha')$. Consider the flat $\text{PSL}(2,{\mathbb C})$--connection
on the projective bundle ${\mathcal P}'$ induced by the
$\text{SL}(2,{\mathbb C})$--connection ${\mathcal D}'(\alpha')$; this induced flat $\text{PSL}(2,
{\mathbb C})$--connection will be denoted by ${\mathcal D}'_P(\alpha')$. From the above
observations it follows that if $\alpha'\, \in\, \widehat{\mathcal U}$ (see \eqref{cux}), then
${\mathcal D}'_P(\alpha')$ descends to a flat $\text{PSL}(2,{\mathbb C})$--connection
on the projective bundle ${\mathcal P}$ over $\Sigma$ constructed in \eqref{pb}.
Let ${\mathcal D}(\alpha')$ denote the flat $\text{PSL}(2,{\mathbb C})$--connection
on ${\mathcal P}$ corresponding to $\alpha'\, \in\, \widehat{\mathcal U}$. The harmonic
reduction of structure group of ${\mathcal P}$ to ${\rm PSU}(2)$ corresponding to
${\mathcal D}(\alpha')$ is evidently the descent of the harmonic reduction of structure
group, to ${\rm PSU}(2)$, of the $\text{PSL}(2,{\mathbb C})$--bundle ${\mathcal P}'$ for the
flat connection ${\mathcal D}'_P(\alpha')$.

Therefore, applying Theorem \ref{thm1} to $(X,\, \widetilde{D})$ we get the following:

\begin{theorem}\label{thm2}
Let $D\,=\,\sum_i n_ix_i$ be an effective divisor on a compact Riemann surface $\Sigma$,
with $g_\Sigma\, \geq\, 2$, such that
\[N\,:=\,\sum_i n_i\,<\,2g_\Sigma-2\, ,\]
and $N$ is an odd integer. Then, there is a unique largest open subset
$$
\widehat{\mathcal U}\, \subset\, H^0(\Sigma,\,K_\Sigma^{\otimes 2}\otimes{\mathcal O}_\Sigma(-2D))
$$
containing $0$ such that for all $\alpha'\, \in\, \widehat{\mathcal U}$,
the harmonic map associated to the ${\rm PSL}(2,{\mathbb C})$--Higgs pair
\[\left({\mathcal P},\;
\begin{pmatrix}0&\alpha'\\ \delta_1&0\end{pmatrix}\right)\]
gives a constant negative curvature $-4$ metric with conical singularities of order $n_i$ at $x_i.$

The above open subset $\widehat{\mathcal U}$ satisfies the condition that
$$
H^0(\Sigma,\,K_\Sigma^{\otimes 2}\otimes{\mathcal O}_\Sigma(-2D-D_{\rm red}))
\, \subset\, \widehat{\mathcal U} \, \subset\,
H^0(\Sigma,\,K_\Sigma^{\otimes 2}\otimes{\mathcal O}_\Sigma(-2D))\, ,
$$
where $D_{\rm red}\,=\,\sum_i x_i$ is the reduced divisor.

When $\alpha'\,=\,0$ the corresponding metric is in the conformal class of $\Sigma$.
\end{theorem}

%%%%%%%%%%%%%%%%%%%%%%%%%%%%%%%%%%%%%%%%%%%%%%%%%%%%%%%%%%%%%%%%%%%%%%%%%%%%%%%%%%%%%%%

\section*{Acknowledgement}

We thank the referee for helpful comments. SB thanks Graeme Wilkin for helpful conversations. IB is partially 
supported by a J. C. Bose Fellowship. SH is supported by the DFG grant HE 6829/3-1 of the DFG priority program SPP 
2026 {\em Geometry at Infinity}. SD has been partially supported by the French government through the UCAJEDI 
Investments in the Future project managed by the National Research Agency (ANR) with the reference number 
ANR2152IDEX201.

\end{document}